\documentclass[a4paper, reqno, 12pt]{amsart}
\usepackage[english]{babel}
\usepackage[T1]{fontenc}
\usepackage{mathptmx,amsmath}
\usepackage{hyperref}
\usepackage{graphicx}
\usepackage{fullpage}
\newtheorem{theorem}{Theorem}[section]
\newtheorem{corollary}[theorem]{Corollary}
\newtheorem{lemma}[theorem]{Lemma}
\newtheorem{definition}[theorem]{Definition}
\newtheorem{proposition}[theorem]{Proposition}
 \theoremstyle{remark}

\numberwithin{equation}{section}

\def\cen{\centerline}
\def\n{\noindent}
\def\al{\alpha}
\def\la{\lambda}
\def\va{\varphi}
\def\cv{\mathcal C(\overline{D})}

\def\C{\mathbb C}

\def\Om{\Omega}
\def\fr{\frac}
\def\F{\mathcal F}
\def\ov{\overline}
\def\om{\omega}

\def\ve{\varepsilon}
\def\de{\delta}
\def\nhd{neighborhood}
\def\no{\noindent}

\begin{document}
\title {Approximation of $m-$subharmonic functions
	 on bounded domains in $\mathbb C^n$}
\author{Nguyen Quang Dieu, Dau Hoang Hung, Hoang Thieu Anh and Sanphet Ounheuan}
\address{Department of Mathematics, Hanoi National University of Education,
136 Xuan Thuy street, Cau Giay, Hanoi, Vietnam}
\email{dieu$\_vn$@yahoo.com.com, dhhungk9@yahoo.com,} 
	
\email {hoangthieuanh@gmail.com, sanphetMA@gmail.com}
\date{\today}
\maketitle
\no
{\bf Abstract.} Let $D$ be a bounded domain in $\mathbb C^n$.
 We study approximation of (not necessarily bounded from above) $m-$subharmonic function $D$ by continuous $m-$subharmonic ones defined on \nhd s of $\ov{D}$. We also consider the existence of a $m-$subharmonic function on $D$ whose boundary values coincides with a given real valued continuous function on $\partial D$ except for a sufficiently small subset of $\partial D.$
 \section{Introduction}
Subharmonic functions and plurisubharmonic functions are fundamental notions in potential theory and pluripotential theory respectively. The theory of $m-$subharmonic ($m-$sh., for short) function was introduced and investigated thoroughly since the seminal work [B\l2]. This new class of functions encompasses the subharmonic and plurisubharmonic ones naturally. The definition of $m-$sh function is however a bit technical.
Let $D$ be an open subset of $\mathbb C^n,$ and $u$ be a subharmonic function defined on $D, u \not\equiv-\infty.$
We say that $u$ is $m-$subharmonic 
($m-$sh. for short) if the $(1,1)$ current $dd^c u$ is $m-$positive in the weak sense, i.e., for 
$\eta_1, \cdots, \eta_{m-1} \in \hat \Gamma_m$ we have
$$dd^c u \wedge \eta_1 \wedge \cdots \wedge \eta_{m-1} \wedge \om^{n-m} \ge 0.$$
Here we define
$$\hat \Gamma_m:=\{\eta \in \mathcal C_{1,1}: \eta \wedge \om^{n-1} \ge 0, \cdots, \eta^m \wedge \om^{n-m}\ge 0\},$$
where $\om=dd^c \vert z\vert^2$ is the standard K\"ahler form and $\mathcal C_{1,1}$
denotes the space of (1, 1)-forms with
constant coefficients. Moreover, $u$ is said to be {\it strictly} $m-$sh. on $D$ if for each relatively compact subdomain $D'$ in $D$ there exists a constant $M>0$ such that 
$u-M\vert z\vert^2$ is $m-$sh. on $D'.$

Thus, in the case $m=1$ or $m=n$ we recover the classes of subharmonic 
and plurisubharmonic functions respectively. 
We write $SH_m(D)$ for the set of $m-$sh functions on $D$.
In this paper, we address the question of approximating an element $u \in SH_m (D)$ by a sequence $\{u_j\}$ of  continuous $m-$sh functions on {\it \nhd s} of $\ov{D}.$ 
If we ask for continuity and $m-$subharmonicity of $u_j$ only on $D$ then our problem has a satisfactory answer if $D$ enjoys certain convexity condition. 
Namely we have the following result which reduces to a classical approximation theorem of Fornaess and Narasimhan in the case where $m=n.$
\begin{theorem}\label{fn}
Assume that $D$ has a $m-$sh. exhaustion function $\va$, i.e., $\{\va<c\}$ is relatively compact in $D$ for all 
$c \in \mathbb R.$ Then for every $u \in SH_m (D)$, there exists a sequence
$\{u_j\}$ of smooth strictly $m-$sh functions on $D$ such that $u_j \downarrow u$ on $D.$
\end{theorem}
On the the hand, under the additional condition that $u$ is {\it bounded} from above on $D$ then the above problem may be approached by the use of Jensen measures associated to certain convex cones in $SH_m (D).$ 
The first results in this direction, in the special case $m=n,$ are due to Wikstr\"om in [Wik] and then in [DW] and [Di]. 
In this work, we will exploit further the techniques developed in [Wik], [DW] and [Di] to work with the case where $u$ is not necessarily bounded from above on $D$. 
For the reader convenience, we will first sketch the general approach in the case 
$u \in SH_m (D)$ with $\sup\limits_D u<\infty.$
Then, after subtracting a large constant we may assume $u<0$ on $D$.
Now let $SH_m^{-} (D)$ be the convex cone of non-positive upper semicontinuous functions on $\ov{D}$ which are $m-$sh on $D$ and $SH_m^* (D)$ be the sub-cone of $SH^{-}_m (D)$ that consists of restrictions on $\ov{D}$
of {\it continuous} $m-$sh functions on \nhd s of $\ov{D}.$
Then, by a general duality theorem of Edwards we may  express upper envelopes of plurisubharmonic functions  taken in $SH^{-}_m (D), SH_m^* (D)$ in terms of  Jensen measures with respect to these cones. The general principle is that the approximation of elements in $SH^{-}_m (D)$ by elements in the smaller cone 
$SH^*_m (D)$ is possible when we have equality of the sets of Jensen measures with respect to these cones. 
In order to formulate our results properly, it is convenient to introduce the following notions pertaining to our work.
\begin{definition}\label{def2}
For a point $z \in \ov D$, we define below two  classes  of  Jensen  measures.
$$\begin{aligned}
&J_{m,z}:= \{ \mu \in \mathcal{B}(\ov D): u(z) \le  \int\limits_{\ov D} ud\mu,  \forall  u \in SH_m^{-}(D)\};\\
&J^c_{m,z}:= \{\mu \in \mathcal{B}(\ov D): u(z) \le  \int\limits_{\ov D} ud\mu,  
\forall  u \in SH_m^* (D)\};
\end{aligned}
$$
where $\mathcal B(\ov D)$ denotes the class of positive, regular Borel measures on $\ov{D}$.
\end{definition}
\n
{\bf Remarks.}
1. If $\xi \in \partial D$ then $J_{m,\xi}=\{\delta_\xi\}.$ 
This is seen by applying Jensen's inequality to the element $u \in SH_m^{-}(D)$ defined by $u(\xi)=0$ whereas $u=-1$ on $\ov{D} \setminus \{\xi\}$.

\noindent
2. For $z \in D$, let $L$ be an affine complex subspace of dimension $n-m+1$ passing through $z, \mathbb B \subset L$ be an open ball centered at $z$ and relatively compact in $L \cap D.$ Then the normalized Lebesgue measure on $\partial \mathbb B$ belongs to 
$J_{m,z}$. This follows directly from Lemma \ref{prop} (h) in the next section and the mean value inequality for subharmonic functions.

\noindent
3. It is obvious that $J_{m,z} \subset J^c_{m,z}$. If $D$ is {\it homogeneous}, i.e., for $p, q \in D$ there exists
an automorphism $\va: D \to D$ sending $p$ to $q$ and extends to a homeomorphism from $\ov{D}$ onto $\ov{D}$, then the set
$\{z \in D: J_{n,z}=J^c_{n,z}\}$ equals either $D$ or the empty set.
\vskip0,2cm
\noindent
The connection between Jensen measures and approximation of $m-$sh functions stems from the following fact which is a simple consequence of Fatou's lemma.
\begin{proposition}\label{pro1}
Let $D \subset \mathbb C^n$ be a bounded domain and $E$ be a subset of $D$. Assume that for every $u \in SH_m^{-}(D)$, 
there exists $\{u_j\}_{j \ge 1} \subset SH_m^{*}(D)$ having the following properties:

\n
(i) $u_j \to u$ pointwise on $E$.

\n
(ii) $\varlimsup\limits_{j \to \infty} u_j \le u$ on $\ov{D}.$

Then $J_{m,z}=J^c_{m,z}$ for every $z \in E.$
\end{proposition}
\noindent
In the opposite direction, the next result  gives a sufficient condition so that
point-wise approximation of functions in $SH_m (D)$ by elements in $SH^*_m (D)$
is possible. We need the following standard notation: If $u: \overline D \to [-\infty,\infty]$ then the upper regularization $u^*$ of $u$ is defined as
$u^*(z):=\varlimsup\limits_{\xi \to z, \xi \in D} u(\xi), \ \forall z \in \ov{D}.$
If $u \in SH_m^{-} (D)$ then obvious ly $u \le u^*$ on $\ov{D}$ while $u=u^*$ on $D.$

\noindent
\begin{theorem}\label{thm1}Let $D \subset \C^n$ be a bounded domain.
Assume that there exists a  subset $X$ of $D$ with $\la_{2n} (X)=0$ such that 
$J_{m,z}=J^c_{m,z} $ for all $z \in D \setminus X.$ Let $E$ be any compact subset of $\partial D$ such that
$J^c_{m,\xi}=\{\de_{\xi}\}, \ \forall \xi \in E.$
Then there exists a $m-$polar subset $Y$ of $D$ having the following property:
For every $u \in SH_m^{-}(D)$, there exists a sequence $\{u_j\} \subset SH_m^*(D)$
having the following properties:

\no 
(i) $u_j$ converges pointwise to $u^*$ on $E \cup (D \setminus Y);$ 

\no 
(ii) $\varlimsup\limits_{j \to \infty} u_j \le u$ on $\ov D$.
\end{theorem}
\n
Here by a $m-$polar set we mean singular the locus of a $m-$sh. function. We postpone to the next section a brief discussion of $m-$polar sets. In the case where $m=m$ and 
$X=\emptyset$, we cover Theorem 3.1 in [DW].

The next theorem, which is our main result, deals with approximation of $m-$sh. functions which are only assumed to be bounded from above on a portion (possibly empty) of $\partial D.$
We will see, at the same time, that 
the exceptional set $Y$ mentioned in Theorem \ref{thm1}  {\it might} occur.
For this purpose, the following piece of notation is useful.
\begin{definition}
Let $a \in D$. Then by $\partial D(a)$ we mean the set of limits points of sequences in 
$\ov{D} \cap h_t (\partial D)$ as $t \uparrow 1,$ where $h_t(z): =t(z-a)+a.$
\end{definition}
\noindent
By an abuse of notation, we will sometimes write $h(t,z)$ instead of $h_t (z).$
Notice that $\partial D(a) \subset \partial D$ and $\partial D(a)=\emptyset$ if and only if 
$\ov{D} \cap h_t (\partial D)=\emptyset$ for $t$ closed enough to $1.$
\begin{theorem}\label{thm2}
Let $D \subset \C^n$ be a bounded domain and $a \in D$. Suppose that there exist an open \nhd\ $U$
in $\mathbb C^n$ of $\partial D(a)$ and a $m-$polar subset $E$ of $U \cap \partial D$  satisfying
the following conditions:

\noindent
(a) $J^c_{m,\xi}=\{\delta_{\xi}\}$ for every 
$\xi \in (\ov{U} \cap \partial D)\setminus E.$ 

\noindent
(b) $J^c_{1,\xi}=\{\delta_{\xi}\}$ for every $\xi \in \ov{U} \cap \partial D.$

Then there exists a $m-$polar subset $E'$ of $D$ such that
for every $u \in SH_m (D)$ satisfying $\sup\limits_{U \cap \partial D} u^*<\infty,$
there exists a sequence $\{u_j\} \subset SH_m^* (D)$ satisfying the following properties:

\noindent
(i) $u_j \to u$ pointwise on $D \setminus E'$;

\noindent
(ii)  $\varlimsup\limits_{j \to \infty} u_j \le u$ on $\ov D;$

\noindent
(iii) $u_j(x) \to u^*(x)$ for every $x \in \partial D$ such that $u$ is continuous at $x$, i.e.,
$$\lim\limits_{z \to x, z \in D} u(z)=u^*(x) \in [-\infty,\infty].$$
\noindent
In particular, if $\partial D(a)=\emptyset$ then $J_{m,z}= J^c_{m,z}$ for all $z \in D.$
\end{theorem}
\noindent
The proof of the above theorem is inspired by Theorem 3.2 in [DW]. Nevertheless, as we will see, besides the (possible) unboundedness from above of $u,$
there is an additional technical difficulty coming from the fact that $1$ may not be a thin point of the segment $t \mapsto tz (0 \le t \le 1)$ for $m-$sh. functions.

\noindent
The structure of Jensen measures is particularly simple at boundary points which admits
a sort of peak $m-$sh. function. We isolate them in the following definition.
\begin{definition}
Let $\xi \in \partial D.$ Then we say that $\xi$ admits a local $m-$sh. barrier if 	there exist a small \nhd\ $U$ of $\xi$ and
$u \in SH_m (U)$ such that $u(\xi)=0$ whereas $u<0$ on $U \cap (\ov{D} \setminus \{\xi\})$.
\end{definition}
\vskip0,2cm
\noindent
{\bf Remarks.} 1. The condition $J^c_{m,\xi}=\{\delta_{\xi}\}$ is fulfilled at $\xi \in \partial D$ 
if there is a local $m-$sh. barrier at $\xi.$
Indeed, by shrinking $U$ we may achieve that 
$\sup\limits_{\ov{D} \cap \partial U} u<0.$ So we may find an open \nhd\ $V$ of $\ov{D} \setminus U$ such that
$\de:=-\sup\limits_{\ov V \cap \partial U}>0$. 
Thus by the gluing lemma (Lemma \ref{prop}(g))
the function
$\tilde u:=\max \{u,-\de\}$ on $U$ and $\tilde u:= -\de$ on $V \setminus U$ is $m-$sh. on $U \cup V$, an open \nhd\ of $\ov{D}$. Now let $\mu \in J^c_{m,\xi}.$ By convolving $u$ with smoothing kernels 
we obtain a sequence $u_j$ of $\mathcal C^\infty-$smooth $m-$sh. functions defined on \nhd s of $\ov{D}$ that decrease to $u$ on $\ov{D}.$ It follows that
$$u_j(\xi) \le \int\limits_{\ov{D}} u_j(z)d\mu.$$
By letting $j \to \infty$ and using Fatou's lemma we conclude that 
$\mu=\{\de_\xi\}.$ This reasoning is essentially contained in Proposition 1.4 in [Si].

\noindent
2. Suppose that there exists an open set $U \subset \mathbb C^n$ and a continuous
{\it strictly} $m-$sh. function $\va$ on $U$ such that $U \cap D=\{z \in U: \va(z)<0\}.$
Then for every $\xi \in U \cap \partial D$ we have $J^c_{m,\xi}=\{\delta_{\xi}\}$.
Indeed, for $M>0$ we set 
$$u_M (z):= M\va(z)-\vert z-\xi\vert^2, \ \forall z \in U.$$
Then for $M>0$ large enough, $u_M$ is a local continuous $m-$sh. barrier at $\xi.$ By the above remark we have $J^c_{m,\xi}=\{\delta_{\xi}\}$.

\noindent
3. Our proof shows that $E'$ is included in the $m-$polar hull of $E$, i.e., intersection of all $m-$polar sets that contain $E$. In particular, if $E=\emptyset$ then $E'=\emptyset$.
\vskip0,2cm
\noindent
In analogy with the concept of $B-$regular domains that was introduced and investigated throughly in [Si] (see also [B\l1]), we 
have the following definition.
\begin{definition}
A bounded domain $D$ in $\mathbb C^n$ is called $B_m-$regular if for every continuous function on $f$ on $\partial D$ we can find a $u \in SH_m (D) \cap \mathcal C(\ov{D})$ such that $u|_{\partial D}=f.$
\end{definition}
\n
If $D$ is $B_m-$regular then obviously for every boundary point $\xi \in \partial D$ there exists $u_\xi \in SH_m (D) \cap \mathcal C(\ov{D})$ such that
$u_\xi (\xi)=$ and $u_\xi (z)<0$ elsewhere.
We will provide a sort of converse to this statement in Theorem 1.10.

Using Theorem \ref{thm1} and Theorem \ref{thm2} we may derive the following consequences.
\begin{corollary}\label{coro}
Let $D$ be an intersection of a finite number of bounded $B_m-$regular domains with $\mathcal C^1-$smooth boundary. 
Then for every $u \in SH_m^{-} (D)$, there exists a sequence $\{u_j\} \subset SH_m^* (D)$ 
such that $u_j \downarrow u$ on $\ov{D}.$
\end{corollary}
\noindent
{\bf Remark.} In the case where $m=n$ and $\partial D$ is $\mathcal C^1-$smooth,
the above result can be deduced by combining results in [Wik] and [FW]. 
Indeed, according to Theorem 4.1 in [Wik]  $u^*$ may be approximated from above on $\ov{D}$ by a decreasing sequence 
$\{v_j\} \subset SH_n (D) \cap \mathcal C(\ov{D})$. It is now suffices to use Theorem 1 in [FW] to approximate each 
$v_j$ uniformly on $\ov{D}$ by elements in $SH_n^*(D)$.
\vskip0,2cm
\noindent
The result below illustrates a class of domains in $\mathbb C^n$ to which Theorem \ref{thm2} is applicable.
\begin{corollary} \label{coro1}
Let $\Om$ be a bounded $B_m-$regular domain in $\mathbb C^n$ and $f$ be a $\mathcal C^1-$smooth function defined on $\Om$. Let
$$D:=\{z \in \Om: f(z)<0\}.$$
For $a \in D$, we set 
$$K_a:=\Big \{z \in \Om \cap \partial D: \Re \Big (z_1-a_1) \frac{\partial f}{\partial z_1}(z)+\cdots+(z_n-a_n) \frac{\partial f}{\partial z_n}(z) \Big) \ge 0 \Big\}.$$
Assume that there exists $a \in D$ such that the following conditions hold true:

\noindent
(i) $f$ is strictly $m-$sh. on an open \nhd\ $U$ of $K_a$. 

\noindent
(ii) There exists a $m-$polar subset $E$ of $U \cap \partial D$ such that for each point
$\xi \in (U \cap \partial D) \setminus E$ we have
$\xi=l_\xi \cap \partial D \cap \Om$, where $l_\xi:=\{t\xi: t \in \mathbb R\}.$

Then $D$ satisfies the condition given in Theorem \ref{thm2}.
\end{corollary}
\n
{\bf Remark.} For a concrete application of the above corollary, consider the case where $n=3, m=2,$
$\Om=\mathbb B_3$ is the unit ball in $\mathbb C^3$ and 
$$f(z_1,z_2,z_3)=\vert z_1\vert^2+\vert z_2\vert^2+\va (\vert z_3\vert^2),$$
where $\va$ is a $\mathcal C^2$ smooth function on $\mathbb R$ satisfying the following conditions:

\n
(a) $\va (0)<0, \va(1)>\va'(1);$

\n
(b) $x\va''(x)+\va'(x)>-\de$ for all $x \in [0,1)$ such that  $x\va' (x) \ge \va(x),$
where $\de \in (0,1/2)$ is a constant;

\n
(c) $\va$ is real analytic on $(0,1);$ 

\n
(d) For every $x \in (0,1)$ there exists $t \in (0,1)$ such that
$\va(tx)\ne t\va(x).$

Then the function $f$ satisfies the conditions (i) and (ii) of Corollary \ref{coro1}.
To see this, we first compute
\begin{equation} \label{ex}
dd^c f=dz_1 \wedge d\ov{z_1}+dz_2 \wedge d\ov{z_2}+
(\va' (\vert z_3\vert^2)+\vert z_3\vert^2 \va''(\vert z_3\vert^2)) dz_3 \wedge d\ov{z_3}.
\end{equation}
In view of $(a)$ we have $a=0 \in D.$ So an easy computation yields that
$$K_a=\Big \{(z_1,z_2,z_3) \in \mathbb B_3: 
\vert z_1\vert^2+\vert z_2\vert^2=-\va(\vert z_3\vert^2), 
\vert z_3\vert^2 \va''(\vert z_3\vert^2) \ge \va' (\vert z_3\vert^2)\}.$$
By (b) and (\ref{ex}) we see that $f$ is strictly $2-$sh. on a small \nhd\ $U$ of $K_a$. 
In view of the second condition in (a), we may obtain that $\vert z_3\vert<1$ on $U$.
Finally, given $\xi=(\xi_1,\xi_2,\xi_3) \in U \cap \partial D$, 
we claim that $\xi$ is an isolated point of 
$l_\xi \cap \partial D.$ If this is false, then there exists a sequence $t_j \to 1$ such that
$f(t_j \xi)=0$. Using the assumption (c) on real analyticity of $\va$ on $(0,1)$ we conclude that $f(t\xi)=0$ for all $t \in (0,1)$. This means that
$$t^2 (\vert \xi_1\vert^2+\vert \xi_2\vert^2)+\va (t^2 \vert \xi_3\vert^2)=0,  \forall t \in (0,1).$$
Plugging $\vert \xi_1\vert^2+\vert \xi_2\vert^2=-\va(\vert \xi_3\vert^2)$ into the above equation
we arrive at a contradiction to (d). The claim follows.

For an example of $\va$ having the above properties we take $\va(x)=\al (c-x)^3,$ where $\al, c$ are real constants such that $\al>0, -2<c<0, \al c^2<\fr2{45}.$
\vskip0,4cm
\n
We end up with the problem of finding a bounded continuous {\it maximal} $m-$sh function $u$ on $D$ 
such that the boundary values of $u$ coincides with a given continuous function defined on part of the boundary 
$\partial D.$
Recall that $u \in SH_m (D)$ is said to be {\it maximal} if for every relatively compact open subset $U$ of $D$ and every $v \in SH_m (D)$ such that $v \le u$ on $D \setminus U$ we have $v \le u$ on $D.$
This definition is analogous to the classical one, $i.e., m=n$ given by Sadullaev (see Proposition 3.1.1 in [Kl]).
\begin{theorem}\label{thm3}
Suppose that there is $v \in SH_m^{-} (D), v> -\infty$ on  $D$ and a compact  $K\subset \partial D$ satisfying the following properties:

\n
(i) $\underset{z \to \xi}\lim v(z)=-\infty,\ \forall   \xi \in K;$

\n
(ii)  Every $\xi \in (\partial D) \setminus K$ admits a local $m$-sh barrier. 

Then for every $\va \in \mathcal C(\partial D)$, there exists uniquely a bounded function
$u \in SH_m (D) \cap \mathcal C(D)$ having the following properties:

\n
(a) $\lim\limits_{z \to \xi, z \in D} u(z)=\va(\xi), \ \forall \xi \in (\partial D)\setminus K;$

\n
(b) $u$ is maximal on $D;$

\n
(c) $u$ can be approximated uniformly on compact sets of $\ov{D} \setminus K$ by elements in 
$SH_m^* (D)$.
\end{theorem}
\noindent
The above theorem appears to be new even in the case where $m=n$ because we allow the existence of points on $\partial D$ which may not admit  continuous plurisubharmonic barriers. Theorem 1.11 also differs somewhat from Theorem 2.1 in [Si] and Theorem 1.7 in [B\l1]
even in the case $K=\emptyset$ and $m=n$, since the solution $u$ may be approximated uniformly on $\ov{D}$ by continuous ones defined on \nhd s of $\ov{D}.$
In the recent preprint [ACH], the authors use Jensen measures to develop some extension and approximation results for $m-$subharmonic functions. They, in particular, generalize several results in [FW] and [Wik] to the context of $m-$subharmonic functions. There is apparently, no overlap between the current paper and their work.
\vskip1cm
\n
\maketitle
\section{Preliminaries}
\no  
Throughout this paper, unless otherwise specified, by $D$ we always mean a bounded domain in $\mathbb C^n.$
We also fix an approximate of identity $\{\rho_\de\}$ in $\mathbb C^n,$
i.e., $\rho_{\de}(x):=\fr1{\de^{2n}} \rho(x/\de),$ where $\rho$ is a smooth radial function with compact support in the unit ball of $\mathbb C^n$ and satisfies $\int\limits_{\mathbb C^n} \rho d\la_{2n}=1$ with $\la_{2n}$ is the Lebesgue measure of $\mathbb C^n.$

Our first lemma contains elementary facts about $m-$sh functions.
The proof of these statements follows from either standard arguments in pluripotential theory (see [Kl]) or from direct computation (see [B\l2]). The details are therefore omitted.
\begin{lemma} \label{prop}
(a) If $u \in SH_m(D)$, then the standard regularization $u_\de:= u*\rho_\de$ is also $m$-sh
in $D_\de:= \{z \in D : d(z,\partial D)>\de\}.$
Moreover, $u_\de \downarrow u$ as $\de \to 0;$

\noindent
(b) If $u,v\in SH_m (D)$ then $au + bv \in SH_m (D)$ for any $a, b \ge 0$, i.e. the class $SH_m(D)$ represents a convex cone;

\no
(c) $PSH(D)= SH_n (D) \subset \cdots \subset SH_1(D) = SH(D)$;

\no
(d) If  $\chi$ is a convex increasing function on $\mathbb R$ and $u \in SH_m (D)$, 
then $\chi \circ u \in SH_m (D)$;

\no
(e) The limit of a uniformly converging or decreasing sequence of $m$-sh functions is m-sh;

\no
(f) The maximum of a finite number of m-sh functions is $m$-sh. More generally, for an arbitrary locally uniformly bounded from above family $\{u_\al\}_{\al \in I} \subset SH_m (D)$ we have
$(\sup\limits_{\al \in I} u_\alpha)^* \in SH_m (D).$

\no
(g) (gluing lemma)
Let  $U$ be an open subset of $D$ such that $\partial U \cap D$ is relatively compact in $D$.
If $u \in SH_m(D), v \in SH_m(U)$ and $\varlimsup\limits_{x \to y} v(x) \le u(y)$ for each 
$y \in \partial U \cap D$, then the
function $w$ defined by
$$w:=\begin{cases}
u & \text{on}\ D\\
\max\{u,v\} &\text{on}\ D \setminus U.
\end{cases}$$
is $m$-sh in D.

\no  
(h) If $u \in SH_m (D)$ then the restriction of $u$ on each $n-m+1$ affine complex subspace is subharmonic.

\n
(i) If $u \in SH_m (D)$ and $g(t)=at+b, t \in \mathbb C^n$ is an affine map then 
$u \circ g \in SH_m (g^{-1} (D)).$ In other words, $m-$subharmonicity is invariant under translations and dilations.

\n
(j) For $1 \le m \le n$ the function $H_m (z)=-\vert z\vert^{2-\fr{2n}m}$ belongs to $SH_m (\mathbb C^n)$.
\end{lemma}
\noindent
Notice, however, that $m-$subharmonicity is not {\it invariant} under composition of holomorphic mappings. 
We now have the useful notion of $m-$polar sets.
\begin{definition}
A subset $E$ of $\mathbb C^n$ is said to be $m-$polar if for every $z_0 \in E$ we may find a \nhd\ $U$ of $z_0$ and
$u \in SH_m (U)$ such that $u=-\infty$ on $E \cap U.$
\end{definition}
The most basic properties of $m-$polar sets are collected below.
\begin{proposition}\label{negligible}
\noindent
(a) For every $m-$polar subset $E$, there exists $u \in SH_m (\mathbb C^n)$ such that $u \equiv -\infty$ on $E$. 

\noindent
(b) Let $\{u_\al\}_{\al \in I}$ be a set of $m-$sh functions defined on $D$ which are locally uniformly bounded from above.
Set $u:=\sup\limits_{\al \in I} u_\al$. Then the set $\{u<u^*\}$ is $m-$polar.

\noindent
(c) Let $\{X_j\}_{j \ge 1}$ be a sequence of $m-$polar sets. Then $\bigcup\limits_{j \ge 1} X_j$ is also $m-$polar.
\end{proposition}
\noindent
In the special case where $m=n$, the above results are proved by Bedford and Taylor using the key notion of relative capacity. The general case can be attacked by the same method where the above notion of capacity is replaced by that of $m-$capacity (see [Lu] or [SA]).
We should say that it is not so easy to construct $m-$polar sets which are not pluripolar ($1-$polar). The following result (Theorem 2.26 in [Lu]) enables us to construct a substantial class of such sets (see Example 2.27 in [Lu]).
\begin{proposition} \label{polar}
Let $H(r):=r^{2n-2m} (1\le m<n)$. Then every subset $E \subset \mathbb C^n$ that satisfies 
$$\infty>\Lambda_H (E):=\lim\limits_{\de \to 0} \Big (\inf \sum_k H(r_k)\Big),$$ 
where the infimum is taken over all coverings of $E$ by balls $B_k$ of radii $r_k \le \delta,$ is $m-$polar.
\end{proposition}
\n
Using the same arguments, we can even show that the class of $m-$ polar sets is properly included in the set of $(m-1)-$polar ones.
The proof of Proposition \ref{polar} uses among other things, a formula for $m-$ relative extremal functions between concentric balls, which requires Lemma \ref{prop} (j). 

\n 
A major technical tool that will be used throughout our work is a version of Edwards' duality theorem which relates upper envelopes of 
upper semicontinuous functions defined on a compact metric space with
lower envelopes of integrals with respect to certain classes of measures. To begin with, let us fix the notation.
Let $X$ be a compact metric space, by $\mathcal C(X)$ we denote the set of real-valued continuous functions on $X.$
We also write $\mathcal B(X)$ for the class of positive, regular Borel measures on $X$.
Let $\mathcal F$ be a convex cone of upper semicontinuous functions on $X$ containing all the constants. If
$g: X \to [-\infty,\infty)$ is a Borel measurable function on $X$ and $z \in X$ then we define
$$Sg(z):= \sup \{u(z) : u \in \mathcal  F, u \le g\},$$
$$Ig(z):= \inf\{\int_X gd\mu:  \mu \in J_z^{\mathcal F} \}.$$
Here $J_z^{\mathcal F}:= \{\mu \in \mathcal B (X): u(z) \le \int_X ud\mu, \ \forall u \in \F\}.$
It is easy to see that $J_z^{\mathcal F}$  is a convex subset of $\mathcal B(X)$. 
Moreover, $\mu (X)=1$ for every $\mu \in J_z^{\mathcal  F}$ since $\mathcal F$ contains the constants.
In view of Banach-Alaoglu's theorem and the fact that every upper semicontinuous function on $X$ is the limit of a decreasing sequence of continuous function on $X$,
we can also check that $J_z^{\mathcal F}$
is weak-$^*$ compact in $\mathcal B(X).$
Now we have the following basic duality theorem of Edwards (see [Ed], [Wik]).
\begin{theorem} \label{edwards}
Let $X, \F$ be as above. If $g: X \to (-\infty, \infty]$ is lower semicontinuous, then $Sg = Ig.$
\end{theorem}
Apparently the first use of Edwards' duality theorem in pluripotential theory has been made in the seminal work  [Si]  where we can find a systematic study of domains in $\C^n$ on which the Dirichlet problem for plurisubharmonic functions is solvable.

In our context, by applying the above theorem to the convex cones $SH_m^{-} (D)$
and $SH_m^{*}(D)$ we obtain the following result which will be referred to as Edwards' duality theorem.
\begin{theorem}\label{thmEd} (Edwards' duality theorem)
Let $\varphi: \ov D \to (-\infty, +\infty]$ be a lower semicontinuous function. Then  we have
$$\begin{aligned}
&\inf \Big \{\int\limits_{\ov D} \varphi d\mu,  \mu \in  J_{m,z}\Big \}= \sup\{u(z): u \in SH_m^{-}(D), 
u \le  \varphi \ \ \text{on}\ \ \ov D \}, \ \forall z \in D\\
&\inf \Big \{\int\limits_{\ov D} \varphi d\mu,  \mu \in  J^c_{m,z} \Big \}=  \sup\{u(z): u \in SH_m^{*}(D), 
u \le  \varphi \ \ \text{ on}\ \ \ov D \},  \ \forall z \in \ov{D}.
\end{aligned}$$
\end{theorem}
Our next ingredients consists of a few standard facts about upper semicontinuous and lower semicontinuous functions on compact sets of $\C^n$.
First, we have an elementary yet useful result of Choquet (see Lemma 2.3.4 in [Kl]).
\n
\begin{lemma}\label{lmC}
Let $\{u_\al\}_{\al \in \mathcal {A}}$ be a family of upper semicontinuous functions defined on a closed subset  
$X \subset \mathbb C^n,$ which is locally bounded from above. Then there exists a countable subfamily $\mathcal {B}$ of $\mathcal A$
such that
$$(\sup \{u_\al: \al \in \mathcal {B}\})^*=(\sup\{u_\al: \al \in \mathcal {A}\})^*.$$
If $u_\al$ are lower semicontinuous then $\mathcal {B}$ can be chosen so that
$$\sup \{u_\al: \al \in \mathcal {B}\}=\sup\{u_\al: \al \in \mathcal{A}\}.$$
\end{lemma}
The next simple lemma deal with monotone sequences of lower semicontinuous on subsets of $\C^n$.
The easy proof is left to the interested reader.
\begin{lemma}\label{lm2}
Let $X$ be a subset  of  $\C^n$  and  $\{\va_j\}_{j \ge 1}$  be a  sequence  of  lower semicontinuous  functions  on  $X$ that increases  to  a lower  semicontinuous  function  $\va$ on $X$.
Then for every  sequence  $\{a_j\}_{j \ge 1} \subset  X$  with  $a_j  \to  a \in X$ we have
$$\va (a) \le \varliminf\limits_{j\to \infty} \va_j (a_j).$$
\end{lemma}
We end up this preparatory section by presenting a useful result which permits approximation of continuous strictly $m-$sh functions by smooth ones. This lemma will be used only in the proof of Theorem \ref{thm1}.
\begin{lemma} \label{rich}
Let $D$ be a domain in $\mathbb C^n.$ Assume that $u$ is a continuous strictly $m-$sh function on $D$. Then for every continuous positive function $h$ on $D$ we can find a smooth strictly $m-$sh function $v$ on $D$ such that 
$u<v<u+h$ on $D$. 
\end{lemma}
\begin{proof} In view of Lemma \ref{prop} we may modify easily the original proof of Richberg's theorem in the case of $m=n$ (see Theorem 1.3 in [B\l1]).
The details are left to the interested reader.
\end{proof}
\section{Proofs of the main results}
\begin{proof} ({\it of Theorem \ref{fn}})
We first show that $D$ admits a {\it smooth} strictly $m-$sh exhaustion function which is larger than $\va$.
This will be done by an adaptation of the proof of Theorem 2.6.11 in [H\"o].
For the reader convenience, we indicate some details. For each $j \ge 1$ we let 
$$\va_j (z):=(\va*\rho_{\de_j})(z)+ \de_j \vert z\vert^2.$$
Here $\de_j>0$ is chosen so small that $\va_j$ is smooth and strictly $m-$sh on $D_{j+1}.$ 
Moreover, we may arrange so that $\va_j>\va$ there.
All this is possible in view of Lemma \ref{prop}.
Take a smooth convex increasing function $\chi$ on $\mathbb R$ such that $\chi(t)<0$ for $t<0$ and $\chi'(t)>0$ when
$t>0.$ Then the function $\chi(\va_j-(j-1))$ is positive, smooth and strictly $m-$sh. on the open set 
$D_{j+1} \setminus \ov{D_j}.$
Therefore we may choose inductively positive numbers $\{a_j\}$ such that
the function 
$$\psi_j:=\sum_{l=1}^j a_l \chi (\va_l+1-l)$$
is strictly $m-$sh and $>\va$ on $D_j.$ By the choice of $\chi$ we also have $\psi_{j'}=\psi_{j''}$ on $D_j$ if $j<j'<j''.$ It follows that $\psi:=\lim\limits_{j \to \infty} \psi_j$ is a smooth, strictly $m-$sh function on $D$. Moreover, since $\psi>\va$, we conclude that  $\psi$ exhausts $D.$ 
Next, in view of Lemma \ref{prop} and Richberg's approximation lemma (cf. Lemma \ref{rich}), we may
repeat the proof of Theorem 5.5 in [FN] to produce the desired approximating sequence for $u$. 
The details are omitted.
\end{proof}
\begin{proof}( {\it of Proposition \ref{pro1}}) Obviously $J_{m,z} \subset J^c_{m,z}, \forall z \in D.$ Conversely,
fix $z \in E$ and $\mu \in J^c_{m,z}.$ For every $u \in SH_m^{*} (D)$ we choose a sequence 
$\{u_j\}_{j \ge 1} \subset SH_m^{*} (D) \cap \cv$ that satisfy the conditions (i) and (ii).
Then we have
$$u_j (z)\le \int\limits_{\ov{D}} u_j d\mu, \ \forall j \ge 1.$$
By letting $j \to \infty$ and making use of Fatou's lemma we get
$$u(z)\le \int\limits_{\ov{D}} ud\mu.$$
Thus $\mu \in J_{m,z}$ as desired.
\end{proof}
For the ease of exposition, we introduce the following notation:
For each bounded function $f$ on $\ov{D},$ we set
\begin{equation}\label{envelop}
\begin{aligned}
&S_m f(z):= \sup\{u(z): u \in SH_m^{-}(D), u^* \le f \ \text{on}\ \ \ov D\}, \  z \in D, \\
&S_m^c f(z):= \sup\{u (z): u \in SH_m^{*}(D), u \le f \ \ \text{ on}\ \ \ov D\}, z \in \ov{D}.
\end{aligned}
\end{equation}
\begin{proof} ({\it of Theorem \ref{thm1}})
We split the proof into two steps.

\no  
{\it Step 1.} We will show that there exists a $m-$polar subset $Y$ of $D$ such that for every $f \in \cv$ we have
$$S_m f=S_m^c f \ \text{on}\ D \setminus Y.$$
Choose a countable dense subset $\{f_j\}$ of $\mathcal C(\ov{D})$. Fix $j \ge 1.$
Then from (\ref{envelop}), Edwards' duality theorem and the fact that $J_{m,z}=J^c_{m,z},$ 
for $z \in D \setminus X,$ we obtain
$$S_m f_j=S_m^c f_j \ \forall z \in D \setminus X.$$
Since $f_j$ is continuous on $\ov{D}$, we have $(S_m f_j)^* \le f_j$ on $\ov{D}$.
Hence $S_m f_j=(S_m f_j)^* \in SH_m (D)$. Since $S_m^c f_j=(S_m^c f_j)^*$ a.e. on $D$ we infer that 
$S_m f_j=(S_m^c f_j)^*$ a.e. on $D$.
Since $S_m f_j$ and $(S_m^c f_j)^*$ are subharmonic on $D$ we deduce that
$$S_m f_j=(S_m^c f_j)^* \ \text{on}\ D.$$
Notice that, by Proposition \ref{negligible}, the set $X_j:=\{z: (S_m^c f_j)^* (z)>S_m^c f_j(z)\}$ is $m-$polar. 
Set $Y:=\bigcup Y_j.$ Then $Y$ is $m-$polar and
$$S_m^c f_j=(S_m^c f_j)^* \ \text{on}\ D \setminus Y.$$
Now we choose a subsequence $\{f_{k_j}\}$ that converges uniformly
to $f$ on $\ov{D}.$ Since $S_m f_{k_j}$ (resp. $S_m^c f_{k_j}$) converges uniformly on $\ov{D}$ to
$S_m f$ (resp. $S_m^c f$), we infer that
$S_m f=S_m^c f$ on $D \setminus Y.$ 

\noindent
{\it Step 2.} We will prove that $Y$ has the properties indicated in the theorem. To this end,
we may assume that $Y$ is $G_\delta.$
Fix $u \in SH^{-}_m (D)$.
We now follow closely the arguments in Theorem 3.1 of [DW].
Choose a sequence of real valued continuous functions $\va_j$ on $\ov{D}$ such that $\va_j \downarrow u^*$ on $\ov D.$
Then by  Edwards' duality theorem and the fact that $J_{m,z}=J^c_{m,z},$ 
for every $z \in D \setminus Y$ we infer
$$S_m^c \va_j =S_m\va_j \  \text{on}\   D \setminus Y.$$
Since $\va_j$ is continuous on $\ov{D}$, we have $(S_m\va_j)^* \le \va_j$.
Therefore
$$S_m\va_j=(S_m\va_j)^* \in SH_m^{-} (D) \ \forall j \ge 1.$$
On the other hand, since $S_m^c \va_j$ is lower semicontinuous on $D$ we deduce that $(S_m \va_j)^*$ is continuous at every point in $D \setminus Y.$
It follows that the {\it restriction} of $S_m^c \va_j$ on $D \setminus Y$ is continuous.
Observe also that $u \le S_m^c \va_j  \le \va_j$ on $D$ for every $j,$ so we get $S_m\va_j \downarrow u$ on $D$.
Hence $S_m^c \va_j \downarrow u$ on $D':=D \setminus Y.$
Since $D'$ is a $F_\sigma$ set, there exists
an exhaustion of $D'$ by compact subsets $\{K_j\}_{j \ge 1}.$ 
By Edwards' duality theorem and the assumption that $J_{m,\xi}=\{\delta_\xi\}$ for each 
$\xi \in E$, we infer that $S_m^c \va_j=\va_j$ on $E.$
In particular $S_m^c \va_j$ is continuous on $K'_j:=E \cup K_j.$
For every $j \ge 1,$ by Choquet's lemma \ref{lmC}, we can find a sequence
$\{v_{l, j}\}_{l \ge 1} \subset SH_m^{*}(D)$ that increases to $S_m^c \va_j$ on $\ov D$. 
By Dini's theorem and continuity of $S_m \va_j$ on $K'_j$, the convergence is uniform on 
$K'_j$ as $l \to \infty$. 
Thus we can choose $v_{l(j), j} \in SH_m^{*}(D)$ such that
$$\Vert S_m\va_j -v_{l(j), j} \Vert_{K'_j} \le \fr1{j}, v_{l(j), j} \le \va_j \  \text{on}\  \partial D.$$
It is then easy to check that $u_j:=v_{l(j), j}$ converges pointwise to $u$ on $E \cup D'$ and
$$\varlimsup\limits_{j \to \infty} u_j \le \lim_{j \to \infty} \va_j=u^* \ \text{on}\ \ov D.$$
The proof is thereby completed 
\end{proof}
\noindent
{\bf Remark.} By the same proof as in the one given in Step 2, we can show that for $z \in D$, the equality $J_{m,z}=J^c_{m,z}$ implies that for each $u \in SH_m^{-}(D)$, 
there exists $\{u_j\}_{j \ge 1} \subset SH_m^{*}(D)$ such that $u_j (z)\to u(z)$ and
$\varlimsup\limits_{j \to \infty} u_j \le u$ on $\ov{D}.$
\begin{proof} ({\it of Theorem \ref{thm2}})
After subtracting a large constant and shrinking $U$ we may assume $\sup\limits_{\ov{U} \cap \partial D} u^*<0$. For $\de>0$ we set
$D_\de:= \{z \in D: \text{dist}\ (z,\partial D)>\de\}.$ 
Fix an exhaustion sequence $\{K_j\}$ of $D$ by compact sets. We claim that for each $j \ge 1$, there exists
$\de_j \in (0,1/j)$ such that 
$$(\partial D) \setminus U \subset h_t (D_{\de_j}), \ \forall t\in (1,1+\de_j].$$ 
Indeed, if the claim is false, then there exists a sequence 
$\{x_m\} \subset (\partial D) \setminus U$ but $x_m \not\in h_{a_m} (D_{1/m})$,
where $a_m:=1+1/m.$ Hence
$$x_m=h_{a_m} (y_m), y_m \not\in D_{1/m}.$$
After switching to a subsequence we may assume that
$\{x_m\} \to x^* \in (\partial D) \setminus U$. It follows that $\{y_m\} \to x^*.$
Now we take a sequence $b_m \uparrow 1$ such that for $m$ large enough we have
$$\vert h_{a_m} (y_m)-h_{b_m} (y^*)\vert<\fr1{m}.$$
It follows that $h_{b_m} (y^*) \to x^*$.
Hence $x^* \in \partial D(a),$ which is a contradiction. The claim follows.
Set 
$$u_j(z):= (u*\rho_{\de_j})(z), z \in D_{\de_j}.$$
We may also choose $\de_j$ such that $\de_j>\de_{j+1}$ and that  $K_j \subset h_t (D_{\de_j})$ for all
$t \in (1,1+\de_j]$. 
Then for $z \in K_j \subset D \cap \psi_t(D_{\de_j})$ and $t \in (1,1+\de_j]$
we have
$$\begin{aligned}
\vert u_j \circ h_t^{-1} (z)- u_j(z)\vert & \le 
\int \vert u_j(x)\vert \vert \vert \rho_{\de_j}(h_t^{-1} (z)-x)-\rho_{\de_j} (z-x)\vert d\la_{2n}(x)\\
&\le M_j\Vert u\Vert_{L^1 (K_j)} \vert t-1\vert.
\end{aligned}$$
Here $M_j>0$ is a constant independent of $t.$
Thus, we can choose $t_j \in (1,1+\de_j]$ such that
\begin{equation} \label{est1}
\Vert u_j \circ h_{t_j}^{-1} (z)- u_j(z)\Vert_{K_j}<\fr1{j}.
\end{equation}
Let $\{\va_j\}$ be a sequence of negative continuous functions on $\partial D$ such that $\va_j \downarrow u^*$ on
$\ov{U} \cap \partial D.$
Let $K$ be a closed ball contained in $D$. Consider the envelopes
$$\begin{aligned}
V(z)&:=\sup \{v(z): v \in SH_m^* (D), v \le -\chi_K\}, z \in \ov{D},\\
\Phi_j (z)&:= \sup \{v(z): v \in SH_m^* (D): v|_{\partial D} \le \va_j \}, z \in \ov{D}.
\end{aligned}$$
Using Edwards' duality theorem and the assumptions (a) and (b) we obtain
\begin{equation} \label{jensen}
\Phi_j=\va_j \ \text{and}\  V=0 \  \text{on}\ \partial D \cap (\ov{U} \setminus E).
\end{equation}
\begin{equation} \label{jensen0}
\Phi_j \le S_1^c \va_j  =\va_j \ \text{on}\ \partial D \cap \ov{U}.
\end{equation}
Since $V^* \in SH_m (D)$ and since $V^*=-1$ on the interior of $K$,
by the maximum principle we have $V \le V^*<0$ on $D.$ Now using Choquet's lemma, we may choose sequences
$\{v_k\} \in SH_m^* (D)$ and $\va_{k,j} \in SH_m^* (D) $
with $\va_{k,j} \uparrow \Phi_j$ and $v_k \uparrow V$ on $\ov{D}.$
Moreover, by the assumption, there exists $\psi \in SH_m^{-}(D')$ with $\psi|_E \equiv -\infty,$ where $D'$ is some open \nhd\ of $\ov{D}.$
Then for fixed $j \ge 1$ we claim that there exist $l_j \ge 1, \de'_j \in (0,\de_j)$ such that for $t \in (1,1+\de'_j]$
we have
$$v_{l_j}(\xi)+\fr1{j} \va_{l_j,j} (\xi) \ge \fr1{j}\varlimsup\limits_{z \to \xi} 
(u_j \circ h_t^{-1})(z)+\fr1{j^2}(\psi*\rho_{\de'_j})(\xi)-\fr1{j^2}\ \ \forall 
\xi \in \ov{D} \cap \va_t (\partial D_{\de'_j}),$$
where $u_j (z):= (u*\rho_{\de'_j})(z), z \in D_{\de'_j}$.
If this is false then we can find sequences $k_l \uparrow \infty, \de'_l \downarrow 0, t_l \downarrow 1$ and 
points $\{\xi_l\}$ such that
$\xi_l \in \ov{D} \cap h_{t_l} (\partial D_{\de'_l})$ and
$$v_{k_l}(\xi_l)+\fr1{j} \va_{k_l,j} (\xi_l)<\fr1{j}(u_{j_l} \circ h_{t_l}^{-1})(\xi_l)+\fr1{j^2}(\psi*\rho_{\de'_j})(\xi_l)-\fr1{j^2}.$$
After passing to a subsequence, we may assume that $\{\xi_l\}$ converges to $\xi^* \in \partial D(a) \subset U.$
On one hand, by Lemma \ref{lm2} we have
$$\varliminf\limits_{l\to\infty} \va_{k_l, j}(\xi_l) \ge \Phi_j(\xi^*) \ \text{and}\ 
\varliminf\limits_{l\to\infty} v_{k_l}(\xi_l) \ge V(\xi^*).$$
On the other hand, 
$$\begin{aligned}
\fr1{j}\varlimsup\limits_{l\to \infty}(u_{j_l} \circ h_{t_l}^{-1})(\xi_l)+\fr1{j^2}\varlimsup\limits_{l\to \infty}(\psi*\rho_{\de'_j})(\xi_l) &\le \fr1{j}u^*(\xi^*)+\fr1{j^2}\psi(\xi^*)\\
<\fr1{j}\va_j (\xi^*)+\fr1{j^2}\psi(\xi^*).
\end{aligned}$$
Putting all this together, using (\ref{jensen}) and the fact that $\psi|_E=-\infty$, we obtain a contradiction. 
The claim is proved.
Furthermore, using Dini's theorem and (\ref{jensen}) again we may choose $l_j$ so large such that
\begin{equation} \label{jensen1}
v_{l_j} \ge -\fr1{j^2} \  \text{on the compact set}\ (\ov{U} \cap \partial D) \setminus \{\psi<-j\}.
\end{equation}
This implies that the function $\tilde u_j$ defined by
$$\tilde u_j:=\begin{cases}
\max \{u_{l_j} \circ h_{\de'_j}^{-1}+\fr1{j}\psi*\rho_{\de_j}-\fr1{j}, jv_{l_j}+\va_{l_j,j}\}&\ \text{on}\ \ov{D} \cap h_{\de'_j} (D_{\de'_j})\\
jv_{l_j}+\va_{l_j,j} & \ \text{on}\ \ov{D} \setminus h_{\de'_j} (D_{\de'_j}),
\end{cases}$$
belongs to $SH_m^* (D)$. Now we set
$$L_j:=\{z \in K_j: u(z) \ge -j, \psi(z) \ge -j\}.$$
Then $L_j$ is a compact subset of $K_j.$ Set $E':=\{z \in D': \psi(z)=-\infty\}.$
Now we claim that $\tilde u_j \to u$ on $D \setminus E'.$
Indeed, given $z_0 \in D$ with $\psi(z_0)>-\infty$. Consider first the case where $u(z_0)>-\infty.$ 
Then $z_0 \in L_j$ for $j$ large enough. 
Since $$jv_{l_j}(z_0) \le jV(z_0) \to -\infty \ \text{as}\ j \to \infty,$$
from (\ref{est1}) we infer that for $j$ sufficiently large
$$\tilde u_j (z_0)=(u_{l_j} \circ h_{\de'_j}^{-1})(z_0)+\fr1{j}()\psi*\rho_{\de_j})(z_0)-\fr1{j}.$$
Therefore, using again (\ref{est1})
and the fact that $u_j (z_0)\downarrow u(z_0)$ 
we see that $\tilde u_j (z_0)\to u(z_0)$ as $j \to \infty.$
On the other hand, if $u(z_0)=-\infty$ then by the same reasoning we have 
$(u_{l_j} \circ h_{\de'_j}^{-1})(z_0) \to -\infty.$ Hence 
$\lim\limits_{j \to \infty}\tilde u_j(z_0)=-\infty=u(z_0).$ This proves (i).
For (ii), we first note that if $z \in \ov{D}$ with $V(z)<0$ then
$$\varlimsup\limits_{j \to \infty} (jv_{l_j} (z)+\va_{l_j,j}(z)) \le 
\varlimsup\limits_{j \to \infty} jV(z)=-\infty.$$
Thus, the preceding proof yields that $\varlimsup\limits_{j\to \infty} \tilde u_j (z)\le u(z).$
For $z \in \ov{D}$ with $V(z)=0$ we have $z \in (\partial D) \cap U.$ It follows, using (\ref{jensen0}) that 
$$\varlimsup\limits_{j \to \infty} (jv_{l_j} (z)+\va_{l_j,j}(z)) \le 
\varlimsup\limits_{j \to \infty} \va_{l_j,j}(z) \le \varlimsup\limits_{j \to \infty} \Phi_j (z) 
\le \varlimsup\limits_{j \to \infty} \va_j (z)\le u^* (z).$$
This proves (ii). Next,
we fix $x \in (\partial D) \setminus E'$ such that 
$u$ is continuous at $x.$ If $x \not\in U$, then $V(x)<0$, so
by the same reasoning as above we get
$$\tilde u_j (x)=(u_{l_j} \circ h_{\de'_j}^{-1})(x)+\fr1{j}(\psi*\rho_{\de_j})(x)-\fr1{j} \to u(x)\  \text{as}\ 
j \to \infty.$$
For the case $x \in U$ we observe that 
$$\varliminf\limits_{j \to \infty} (jv_{l_j} (x)+\va_{l_j,j}(x))=
\varliminf\limits_{j \to \infty} \va_{l_j,j}(x) \ge 
\varliminf\limits_{k \to \infty}(\varliminf\limits_{j \to \infty}  \va_{k,j}(x))=u(x).$$
Hence we get (iii).
Finally, we note that if $u \in SH_m^{-} (D)$ then so is $\tilde u_j$. Thus, if $\partial D(a)=\emptyset$ we may choose $U=\emptyset.$ It follows that 
$\varlimsup\limits_{j\to \infty} u_j \le u$ on $\ov{D}$. Hence, by applying Proposition \ref{pro1} we get that
$J_{m, \xi} (D)=J^c_{m, \xi}$  for all $\xi \in D.$
\end{proof}
\vskip0,2cm
\noindent
For the proof of Corollary \ref{coro} we need the following lemma
\begin{lemma} \label{localpeak}
Let $D$ be as in Corollary \ref{coro} and $\xi \in \partial D$. Then
$J^c_{m,\xi}(D)=\{\de_\xi\}.$ 
\end{lemma}
\begin{proof}
We split the proof into two steps.

\no 
{\it Step 1.} We first show that $J^c_{m,\xi} (D \cap U)=\{\de_\xi\}$ for some small \nhd\ $U$ of $\xi.$ 
To see this, we write
$D=D_1 \cap \cdots \cap D_k$, where $D_1, \cdots, D_k$ are $B_m-$regular domains with 
$\mathcal C^1-$ smooth boundaries.
Then $\xi \in \partial D_j$ for some $1 \le j \le k.$
Since $D_j$ is $B_m-$regular, we can find $u \in SH_m (D_j) \cap \mathcal C(\ov{D_j})$
such that $u(\xi)=-1$ and $u<-1$ on $\ov{D_j} \setminus \{\xi\}.$ 
Since $\partial D_j$ is $\mathcal C^1-$smooth, we may find $\ve>0$ and
open balls $\mathbb B_1 \subset \mathbb B_2$ such that
$$\ov{\mathbb B_1 \cap D_j} \subset (\mathbb B_1 \cap D)+\ve \bold n, \ \forall \ve \in (0,\ve_0),$$
where $\bold n$ is the unit outward normal to $\partial D$ at $\xi.$
We claim that $J^c_{m,\xi} (D')=\{\de_\xi\}$, where $D':=D_j \cap \mathbb B_1.$
For this, we 
set $u_\ve(z):= u(z-\ve \bold n).$ Then $u_\ve$ is $m-$sh on a \nhd\ of $\ov{D'}$ for each $\ve \in (0, \ve_0).$ 
Fix $\mu \in J^c_{m,\xi} (D').$ Then for $\de>0$ small enough we have
$$(u_\ve *\rho_\de) (\xi) \le \int\limits_{\ov{D'}} (u_\ve *\rho_\de)\mu.$$
By letting $\de \downarrow 0$ and $\ve \downarrow 0$ and using Fatou's lemma we obtain
$-1 \le \int\limits_{\ov{D'}} ud\mu.$
This forces $\mu=\de_\xi$ as claimed. 

\no 
{\it Step 2.} Fix $\nu \in J^c_{m,\xi}(D)$ we will show that $\nu=\{\de_\xi\}.$
Consider the upper envelope
$$\va(z):=\sup\{u(z): u \in SH_m^c (D'): u \le h \ \text{on}\ \ov{D'}\},$$
where $h(z):=-\vert z-\xi\vert$. Then by Edwards' duality theorem and the result obtained in the previous step we obtain $\va(\xi)=0$ whereas $\va<0$ on $\ov{D'} \setminus \{\xi\}.$
Thus, by Choquet's lemma, we can find a sequence $\{u_k\} \subset SH_m^c (D')$ such that $u_k \uparrow \va$ as $k \to \infty.$ In particular $u_k \le h$ on $\ov{D'}.$ 
Hence, we can find $\de>1$ such that 
$u_k<-\de$ on $\partial D' \cap D.$ Define for each $k$ the function
$$\tilde u_k:=\begin{cases}
-\de \ & \text{on}\ \ov{D} \setminus D'\\ 
\max \{u_k, -\de\} & \text{on}\ D'.
\end{cases}$$
By Lemma \ref{prop}, we see that $\tilde u_k \in SH_m^c (D)$. Thus we have
$$\tilde u_k (\xi) \le \int\limits_{\ov{D}} u_kd\nu.$$
By letting $k \to \infty$ and using Lebesgue monotone convergence theorem we infer that $\nu=\{\de_\xi\}.$
This completes the proof of the lemma.
\end{proof}
\begin{proof} {(\it of Corollary \ref{coro})}
In view of the above lemma, we may apply Theorem \ref{thm2} to $U=\mathbb C^n, E=\emptyset$ 
and $a$ is an arbitrary point of $D$ to get the desired conclusions.
\end{proof}
\begin{proof} {(\it of Corollary \ref{coro1})}
First we show that $\partial D(a)$ is included in $K_a.$
Indeed, fix $\al \in \partial D(a).$ Then we can find a sequence $t_j \uparrow 1$ and a sequence $p_j \in \partial D$ such that 
$$\ov{D} \ni \al_j:=h(t_j, p_j)\to a.$$ 
Recall that $h_t (t,z)=h_t (z)=t(z-a)+a.$ 
Set $t'_j:=1/t_j>1.$
It follows that  $g(h(t'_j,\al_j))=0$ whereas $g(h(1,\al_j)) \le 0.$ Hence, there exists 
$t''_j \in (1, t'_j)$ such that 
$$\begin{aligned}
0 &\le \frac{\partial}{\partial t} f(h(t,\al_j))\big \vert_{t=t''_j}\\
&=2 \Re \Big ((\al_1-a) \frac{\partial f}{\partial z_1}\circ h(t''_j, \al_1)+\cdots+(\al_n-a) \frac{\partial f}{\partial z_n}\circ h(t''_j, \al_n)\Big).
\end{aligned}$$
By letting $j \to \infty$ we see that $\al \in K_a.$
Next, fix $\xi=(\xi_1,\cdots,\xi_n) \in (U \cap \partial D) \setminus E$, 
we will show that $J^c_{m,c} (\xi)=\{\de_\xi\}.$
To see this, it suffices to construct a {\it local} $m-$sh barrier of $D$ at $\xi.$
Since $\Om$ is $B_m-$regular and $\partial \Om$ is $\mathcal C^1-$smooth, by Lemma 3.1 we may assume that $\xi \in \Om.$ Since $f$ is strictly $m-$sh. on $U$, after shrinking $U$ and multiplying $f$ with a constant we may assume
$$f(z)=\vert z\vert^2 +g(z),$$
where $g \in SH_m (U) \cap \mathcal C^1 (U).$ Notice that $g(\xi)=-\vert \xi\vert^2<0.$
For $z \in U$ we define
$$u_\xi (z)=\Re (z_1^2\ov{\xi_1^2}+\cdots+z_n^2\ov{\xi_n^2})-g(z)g(\xi).$$
By the hypothesis, 
we see that $u_\xi$ is continuous $m-$sh on $U$ and $u_\xi (\xi)=0$.
Moreover, for $\tilde \xi \in \ov{D} \setminus \{\xi\},$
by Cauchy-Schwarz's inequality we obtain
$$u_\xi (\tilde \xi) \le \vert \tilde \xi\vert^2 \vert z\vert^2-g(z)g(\tilde \xi) \le 0.$$
Here the equality occurs only if $\tilde \xi \in \partial D$ and 
$\tilde \xi=t\xi$ for some constant $t \ge 0, t \ne 1,$ i.e., $\tilde \xi \in l_\xi.$
This is impossible in view of (ii).  Hence, $u_\xi$ is indeed a local $m-$sh. barrier at $\xi$.
The proof is thereby completed.
\end{proof}
\noindent
The proof of Theorem \ref{thm3} requires the following fact.
\begin{lemma} \label{alpha}
Let $\xi \in \partial D$ be a boundary point that admits
a local $m-$sh. barrier.
Let $\va<0$ be a continuous function on $\partial D.$
Then for every sequence $\{\xi_j\} \subset D$ with $\xi_j \to \xi$  we have
$$\varlimsup\limits_{j \to \infty} S_m^c \va (\xi_j) \le \va (\xi),$$
where $\va$ is extended to a lower semicontinuous function on $\ov{D}$ by setting $\va:=+\infty$ on $D.$
\end{lemma}
\begin{proof} Let $u$ be a local $m-$sh. barrier at $\xi.$
By the argument given in the remark following Theorem \ref{thm2}, we may extend $u$ to a $m-$sh. function $\tilde u$ on a \nhd\ of $\ov{D}$ such that $\tilde u(\xi)=0$ while $\tilde u<0$ on 
$\ov{D} \setminus \{\xi\}.$
Let $\{\mu_j\}_{j \ge 1}, \mu_j \in J_{m,z_j}$ be a sequence of Jensen measures with compact support
in $\partial D.$ We claim that $\mu_j$ converges to $\delta_\xi$ in the weak $^*-$ topology.
Let $\mu^*$ be a cluster point of $\{\mu_j\}$. Then for $\de>0$ small enough and $j \ge 1$ we have
$$(\tilde u *\rho_\de)(z_j) \le \int\limits_{\partial D} (\tilde u *\rho_\de)d\mu_j.$$
By letting $\de \downarrow 0$ and then $j \to \infty$ we infer that
$\mu^* =\delta_\xi.$ This proves the claim.
It follows that
$$\varlimsup\limits_{j \to \infty} S_m^c \va (z_j) \le \varlimsup\limits_{j \to \infty} \int_{\partial D} \va d\mu_j \le \va(\xi).$$
This is the desired conclusion.
\end{proof}
\begin{proof}  ({\it of Theorem \ref{thm3}})
We split the proof in two two parts.

\n
{\it Existence.} After subtracting a large constant we may assume $\va<0$ on $\partial D.$ Define 
$$\begin{aligned}
&S_m\va (z):=  \sup\{u(z): u \in SH_m^{-}(D), u \le \varphi \ \ \text{on}\ \ \partial D\}, \  z \in D; \\
&S_m^c \va (z):= \sup\{u(z): u \in SH_m^{*}(D), u \le  \varphi \ \ \text{on}\ \  \partial D\}, z \in \ov{D}.
\end{aligned}$$
Then by Lemma 2.1 (f),  $u:=(S_m^c \va)^* \in SH_m^{-} (D)$.
In view of the assumption (a) and the remark following Theorem \ref{thm2} we also have $J^c_{m,\xi}=\{\de_\xi\}$ for every 
$\xi \in (\partial D) \setminus K.$
So using Edwards' duality theorem (with $\va:=+\infty$ on $D$) we obtain
\begin{equation}\label{i1}
S_m^c \va =\va \  \text{on}\ ( \partial D) \setminus K.
\end{equation}
Furthermore, by Theorem \ref{thm2} we get $J_{m,z}=J^c_{m,z}$ for every $z \in D.$ 
So applying again Edwards' duality theorem we get
\begin{equation}
S_m\va =S_m^c \va \ \text{on}\  D.
\end{equation}
On the other hand, by the assumption (ii) and Lemma \ref{alpha} 
we get
\begin{equation} \label{i2}
\varlimsup\limits_{z \to \xi, z \in D} u(z) \le \va(\xi), \ \forall \xi \in (\partial D)\setminus K.
\end{equation}
Fix $\ve>0$ and set $u_\ve:= u+\ve v.$
Then we infer from the last inequality and the assumption that $(u_\ve)^* \le \va$ on $\partial D.$ This implies that $u_\ve \le S_m \va$ on $D.$
By letting $\ve \downarrow 0$ and noting that $v>-\infty$ on $D$ we get
$$S^c_m \va \le u \le S_m \va=S_m^c \va \ \text{on}\  D.$$ 
Hence $u=S_m^c \va$ on $D$. In particular $u$ is lower semicontinuous on $D$. Therefore
$u \in SH_m(D) \cap \mathcal C(D)$ and $\Vert u\Vert_D \le \Vert \va\Vert_{\partial D}.$
Next we claim that $u$ has the right boundary values off $K$. Indeed, fix $\xi \in (\partial D) \setminus K$ and a 
sequence $\xi_j \to \xi, \xi_j \in D.$
By the lower semicontinuity on $\ov{D}$ of $S_m^c \va$ and (\ref{i1}), (\ref{i2}) we have
$$\va (\xi) =S_m^c \va(\xi) \le \varliminf\limits_{j \to \infty} u(\xi_j) \le  \varlimsup\limits_{j \to \infty} u(\xi_j) \le \va(\xi) .$$
It follows that $\lim\limits_{z \to \xi}u(z)=\va(\xi).$ This proves our claim. 
By the same argument as the one given at the end of the proof of Theorem 1.4 (applying Choquet's lemma and then Dini's theorem on an exhaustion sequence of compact subsets of $\ov{D} \setminus K$) we conclude that 
$u$ may be approximated uniformly on compact sets of $\ov{D} \setminus K$ by elements in $SH_m^*(D).$
Finally, for maximality of $u$, let $w \in SH_m (D)$ with $w \le u$ on $D \setminus U$ for some open set $U$ relatively compact in $D.$
Then the function
$$\tilde u(z):= \begin{cases}
\max \{u(z),  w(z)\},  & z \in U \\
u(z) & z \in D \setminus U
\end{cases}$$
belongs to $SH_m(D)$. Moreover, $\tilde u$ is a member in the defining family for $S_m\va$. Therefore
$\tilde u \le S_m \va =u$ on $D$. In particular, $w \le u$ on $U.$  This proves the existence of the desired solution.

 \n
{\it Uniqueness.} Assume that there exist bounded continuous maximal functions 
$u_1, u_2 \in SH_m (D)$ on $D$ such that
$$\lim\limits_{z \to \xi, z \in D} u_1(z)=\lim\limits_{z \to \xi, z \in D} u_2 (z)=\va (\xi), \ \forall \xi \in (\partial D)\setminus K.$$
Let $\{D_j\}_{j \ge 1}$ be a sequence of relative compact open subset of $D$ with $D_j \uparrow D.$
Fix $\ve>0$, since $u_2$ is bounded from below on $D$
we can find $j_0 \ge 1$ so large such that
$$u_1+\ve v-\ve \le u_2 \ \text{on}\  D \setminus  D_{j _0}.$$
It follows  from maximality of $u_2$ that
$$u_1+\ve v-\ve \le u_2 \ \text{on}\   D.$$
By letting $\ve \downarrow 0,$ we infer that $u_1 \le u_2$ on $D$. 
Similarly we also get $u_2 \le u_1$ on $D$. Therefore $u_1=u_2$ on $D.$

The theorem is proved.
\end{proof}
\n
{\bf Remark.} If we do not assume that $v>-\infty$ on $D$ then a slight modification of the above proof (similar to the proof of Theorem 1.3)  gives a maximal $m-$sh function $u$ on $D$ with boundary values $\va$ (on $(\partial D)\setminus K$ which is only continuous at every point $z \in D$ with $v(z)>-\infty.$
\vskip1,5cm
\cen {\bf References}

\noindent
[ACH] P. Ahag, R. Czyz and L. Hed, {\it Extension and approximation of $m-$subharmonic functions}, preprint
https://arxiv.org/pdf/1703.03181.pdf.

\n
[B\l1] Z. B\l ocki, {\it The complex Monge-Amp\`ere operator in hyperconvex domains},
Annali della Scuola Normale Superiore di Pisa, {\bf 23} (1996), 721-747.

\noindent
[B\l2] Z. B\l ocki, {\it Weak solutions to the complex Hessian equation}, Ann. Inst. Fourier (Grenoble)
{\bf 55} (2005) 1735-1756.

\n [Ed] D. A. Edwards, {\it Choquet boundary theory for certain spaces of lower semicontinuous
functions}, in Function Algebras (Proc. Internat. Symposium on Function Algebras,
Tulane Univ., 1965) (Scott-Foresman, Chicago, 1966), pp. 300-309.

\n
[DW]  N.Q. Dieu and F. Wikstr\"om, {\it Jensen measures and approximation of plurisubharmonic functions}, Michigan Math. J. {\bf 53} (2005) 529-544.

\n
[Di] Nguyen Quang Dieu, {\it Approximation of plurisubharmonic functions on bounded domains in $\mathbb C^n$}, Michigan Math. J. {\bf 54}, (2006) 697-711.

\n
[FN] J. E. Fornaess and R. Narasimhan, {\it The Levi problem on complex spaces with singularities},
Math. Ann. {\bf 248} (1980) 47-72.

\noindent
[Kl] M. Klimek, {\it Pluripotential Theory}, Oxford 1991.

\noindent
[Lu] C.H. Lu, {\it A variational approach to complex Hessian equations in $\mathbb C^n$}, J. Math. Anal. Appl. 
{\bf 431} (2015), 228-259.

\noindent
[SA] A. Sadullaev and B. Abdullaev, {\it Potential theory in the class of m-subharmonic},  Proc. Steklov Inst. Math. 
{\bf 279} (2012), 155-180.  

\n
[Si] N. Sibony, {\it Une classe de domaines pseudoconvexes}, Duke Math. J. {\bf 55} (1987), 299-319.

\n
[Wik]  F. Wikstr\"om, {\it Jensen measures and boundary values of plurisubharmonic functions,}
Ark. Mat. {\bf 39} (2001) 181-200.

\end{document}